\documentclass{article}

\def\be{\begin{equation}}
  \newcommand{\lab}[1]{\label{#1}}

  \usepackage[top = 3.4cm, bottom = 3.4cm, left = 3.4cm, right = 3.4cm]{geometry}

  \usepackage[english]{babel}
  \usepackage{amsfonts,amsmath,amssymb,amsthm}
  \usepackage{nicefrac}
  \usepackage{enumerate}
  \usepackage{url}
  \usepackage{ifthen}
  \usepackage{verbatim}
  \usepackage{mathrsfs}
  \usepackage{textcomp}

  \newcommand{\setN}{\ensuremath{\mathbb{N}}}
  
  \newcommand{\setR}{\ensuremath{\mathbb{R}}}

  \newcommand{\Dcal}{\ensuremath{\mathcal{D}}}

  \newcommand{\eqdef}{:=}

  \newcommand{\set}[2][]{#1\{ {#2} #1\}}
  \newcommand{\card}[2][]{#1| #2 #1|}
  \newcommand{\ceil}[2][]{#1\lceil #2 #1\rceil}

  \newcommand{\st}{\colon}

  \newcommand{\prob}[2][]{\mathbb{P}#1(#2#1)}

\newcommand{\pr}{\mathbb{P}}

  \newcommand{\esp}[2][]{\mathbb{E}\,#1(#2#1)}

  \newcommand{\eps}{\varepsilon}
  \newcommand{\tc}{$2$-connected}

  \newcommand{\tcnm}{\ensuremath{T(n,m)}}
  \newcommand{\tcd}{\ensuremath{T(\ds)}}
  \newcommand{\cnm}{\ensuremath{C(n,m)}}

  \newcommand{\tec}{$2$-edge-connected}
  \newcommand{\aas}{a.a.s.}
  \newcommand{\Aas}{A.a.s.}
  \newcommand{\ds}{\mathbf{d}}
  \newcommand{\Ys}{\mathbf{Y}}
  \newcommand{\tcs}{\textbf{2cs}}

  \newcommand{\tpoisson}[3][]{\mathop{\rm Po}#1(#2, #3 #1)}
  \newcommand{\poisson}[2][]{\mathop{\rm Po}#1(#2 #1)}

  \newtheorem{thm}{Theorem}
  \newtheorem{lem}[thm]{Lemma}
  \newtheorem{prop}[thm]{Proposition}
  \newtheorem{cor}[thm]{Corollary}

  \newlength{\topsepBackup}
  \newlength{\mathindent}
  \newlength{\halfmathindent}
  \newlength{\lassertion}
  \newcounter{countercl}

  \newcounter{mylistcount}
    {%
    \end{list}%
    \smallskip
  }

  \newlength{\classertion}
  \newcommand{\myreference}{}

  {%
    \vspace*{-5pt}
    \begin{flushright}
      \textit{End of Proof of Claim~\myreference}
    \end{flushright}
    \smallskip
  }

    {%
    \end{list}%
    \smallskip
  }

  \newcommand{\Luczak}{\L uczak}

\if01
   \newcommand{\graeme}[1]
  {{\bf [ Graeme: } {\em #1}{\bf ]}}
   \newcommand{\cris}[1]
  {{\bf [ Cris: } {\em #1}{\bf ]}}

  \newcommand{\crisa}[1]
  {{\bf [ Cris 4 Oct: } {\em #1}{\bf ]}}

  \newcommand{\crisb}[1]
  {{\bf [ Cris 5 Oct: } {\em #1}{\bf ]}}
    
  \newcommand{\crisc}[1]
  {{\bf [ Cris 7 Oct: } {\em #1}{\bf ]}}
  
  \newcommand{\crisd}[1]
  {{\bf [ Cris 9 Oct: } {\em #1}{\bf ]}}

  \newcommand{\crise}[1]
  {{\bf [ Cris 11 Oct: } {\em #1}{\bf ]}}
 
\newcommand{\nicka}[1]
  {{\bf [ Nick 4 Sept: } {\em #1}{\bf ]}}

  \newcommand{\nickb}[1]
  {{\bf [ Correction } {\em #1}{\bf ]}}
\fi
\if01 
\renewcommand{\cris}[1]{}
  \renewcommand{\crisa}[1]{}
 \renewcommand{\crisb}[1]{}
 \renewcommand{\graeme}[1]{}
\renewcommand{\nicka}[1]{}
\renewcommand{\nickb}[1]{}
 \renewcommand{\nickc}[1]{}
\renewcommand{\nickd}[1]{}
\renewcommand{\nickcc}[1]{}
\renewcommand{\nicke}[1]{}
 \renewcommand{\nickf}[1]{}
\fi

  \newcommand{\eqn}[1]{(\ref{#1})}
  \newcommand{\bel}[1]{\be\lab{#1}}
  \def\ee{\end{equation}}

\title{Asymptotic enumeration of sparse 2-connected graphs}
\author{Graeme Kemkes\thanks{This author's research was primarily carried out while at the University of California at San Diego under an NSERC postdoctoral award.}\\
Dept.\ of Mathematics\\
Ryerson University\\
Toronto, ON\\
Canada, M5B 2K3\\ 
\small{E-mail: {\tt gdkemkes@ryerson.ca} }
\and
Cristiane M.~Sato\\
Dept.\ of Combinatorics and Optimization\\
University of Waterloo\\
Waterloo ON\\
Canada N2L 3G1\\
\small{E-mail: {\tt  cmsato@gmail.com}}
\and
Nicholas Wormald\thanks{Supported by the  Canada Research Chairs Program and NSERC.}
\\
Dept.\ of Combinatorics and Optimization\\
University of Waterloo\\
Waterloo ON\\
Canada N2L 3G1\\
\small{E-mail: {\tt   nwormald@uwaterloo.ca}}
}

\date{}

\begin{document}

\maketitle

\begin{abstract}
  We determine an asymptotic formula for the number of labelled
  2-connected  (simple) graphs on $n$ vertices and $m$ edges,
  provided that $m-n\to\infty$ and $m=O(n\log n)$ as
  $n\to\infty$. This is the entire range of $m$ not covered by
  previous results.  The proof involves determining properties of the
  core and kernel of random graphs with minimum degree at least 2. The case of 2-edge-connectedness is treated similarly. We
  also obtain formulae for the number of 2-connected graphs with given
  degree sequence for most (`typical') sequences.  Our main result
  solves a problem of Wright from 1983.
 \end{abstract}


\section{Introduction}
Counting graphs with a given property is a fundamental and often
difficult problem.  G.E. Uhlenbeck, in the Gibbs Lecture at an
American Mathematical Society meeting in 1950, cited the enumeration
of 2-connected graphs as one of the unsolved problems in statistical
mechanics. In the ensuing years, ways were found to efficiently
calculate the number of such graphs with a given number of vertices,
or vertices and edges (see Harary and Palmer~\cite{HP} for
example). However, no very simple formula was found, which brings up
the question of asymptotic formulae. In the two-parameter case, there
are some ranges of the parameters for which such a formula is
unknown. This is the subject of the present paper.

Call a  (simple) graph on the vertex set $[n]=\set{1,\ldots, n}$
with $m$ edges an $(n, m)$-graph. (Thus, we are concerned with
labelled graphs.) A well-studied problem is to count $(n,m)$-graphs
with minimum degree at least some fixed number, $k$.
Korshunov~\cite{Korshunov}, and Bender, Canfield and McKay~\cite{BCMb}
provided asymptotic formulae for the case $k=1$, that is, graphs with
no isolated vertices. The case $k=2$, which is relevant for
2-connected graphs, was studied by Wright~\cite{Wright78} and others
such as Ravelomanana and Thimonier~\cite{RT}. Later, Pittel and
Wormald~\cite{PWa} found asymptotic formulae for the case $k\geq 2$.

A number of authors have addressed the problem of counting connected
$(n,m)$-graphs. After results by various authors for various ranges of
$m$ with various degrees of approximation, Bender, Canfield and
McKay~\cite{BCMa} provided an asymptotic formula for the number
whenever $m-n\to\infty$ as $n\to \infty$. They obtained this formula
by studying a differential equation related to a recurrence relation
for the number of connected graphs. Pittel and Wormald~\cite{PWb}
provided a somewhat simpler proof for this formula, with an improved
error term for some ranges of $m$.

A natural next step would be to count $k$-connected
$(n,m)$-graphs. This problems turns out to be essentially already
solved for $k\geq 3$. \Luczak~\cite{Luczak} showed that a random graph
with given degree sequence, all degrees between 3 and $d$, a.a.s.\ has
connectivity equal to minimum degree. As observed in the introduction
of~\cite{PWa}, this implies that, for $m=O(n\log n)$, a random
$(n,m)$-graph with minimum degree $k\geq 3$ is \aas{}
$k$-connected. (To deduce this, one needs to know that such a random
$(n,m)$-graph has no large degree vertices, which can be deduced from
the results of~\cite{PWa}, or alternatively by a more direct argument
if $m/n$ is bounded.)  Thus, using the above-mentioned result
from~\cite{PWa}, one immediately obtains an asymptotic formula for
$k$-connected $(n,m)$-graphs. However, this argument does not apply
for $2$-connected graphs.

In this article, we derive an asymptotic formula for the number of
2-connected $(n, m)$-graphs when $m-n\to\infty$ with $m = O(n\log
n)$. Above this range, for any fixed $k$, it is well known that almost
all graphs are $k$-connected. This follows by the classic result of  Erd{\H o}s and R{\'e}nyi~\cite{ER}, that  for fixed $k\ge 0$ and $m=m(n)= {1 \over 2} n(\log n + k
\log \log n+x+o(1))$,  
$$
\pr( {\cal G}(n,m)\mbox{ is $k$-connected})\to 1-e^{-e^{-x}/k!},
$$ 
where ${\cal G}(n,m)$ denotes an $(n,m)$-graph chosen uniformly at random.
Wright~\cite{Wright83} found an asymptotic formula for the case $m-n = o(\sqrt{n})$ with
$m-n\to\infty$, and it was noted that the problem of finding a formula
for $m-n$ growing faster than $\sqrt{n}$ seems difficult. We solve
this problem here. 

Regarding exact enumeration, Temperley~\cite{Temp} proved a recurrence
relation for the number of $2$-connected $(n,m)$-graphs. His proof
uses calculus to deduce the recurrence formula from a well known
differential equation for the generating functions of the number of
connected graphs and $2$-connected graphs; see~\cite{HP} (pp.~10, 11).
For a combinatorial proof, see~\cite{WW}. Wright~\cite{Wright78} also
found an exact formula for the number $2$-connected $(n,n+k)$-graphs
with fixed $k$.  It is possible that following an approach close to
the one in~\cite{BCMa}, one could obtain an asymptotic formula for the
\tc{} $(n,m)$-graphs.  However, we believe that this would be very
difficult since the proof in~\cite{BCMa} is not simple and the
recurrence relation for \tc{} graphs is more complicated than the one
for connected graphs.

The $k$-core of a graph $G$ is a maximal subgraph of $G$ with
minimum degree at least $k$.  We extend this definition and simply
call a graph a {\em $k$-core} if it has minimum degree at least $k$.

Since $2$-connected graphs must have minimum degree at least $2$ (if
they have more than two vertices), we work with $2$-cores. Our
approach uses some of the basic ideas in~\cite{PWb} where the
(sub-)problem being addressed was asymptotic enumeration of connected
2-cores with a given number of edges and vertices. One possible plan
can be described as follows, and could potentially be of use for any
graph property.  First, compute the probability that a random graph
with given degree sequence is 2-connected, where the degree sequence
is chosen randomly, the degrees being independent truncated Poisson
random variables, conditioned on the sum being $2m$. (Truncated means
conditioning on the value being at least 2.) Next, using the results
in~\cite{PWa} we can try to deduce that the same probability of
connectedness holds for a random 2-core. In that paper, it is
essentially shown, starting with the basic enumeration results of
McKay~\cite{McKay}, that the degree sequence of a random $(n,m)$-graph
which is a 2-core is strongly related to a sequence of independent
truncated Poisson random variables, conditioned on the sum being $2m$
(see~\cite[Equation (13)]{PWa} for example).  Note that, in the first
step, we do not need to estimate the probability for some degree
sequences (e.g., if the maximum degree is too high) as long as we show
that they have very low probability of occurring as the degree
sequence of a random $(n,m)$-graph with minimum degree at least~$2$.

This plan works quite well provided $m/n\to\infty$, in which case we use another result of \Luczak~\cite{Luczak}  to show that the probability of 2-connectedness tends to 1 in the Poisson-based model. For this we need to use the pairing model, a common model used for analysing random graphs with given degrees. However, if $m/n \to 1$, a random 2-core tends to have many isolated cycles, so the probability of 2-connectedness tends to 0 and then the   plan is difficult to carry out. For such $m$, and, for convenience, whenever $m$ is bounded, we use a construction in~\cite{PWb} called the kernel configuration model, a modification of the pairing model. This is a probability space enabling direct analysis of the 2-cores that have no isolated cycles, and the above plan is readily adapted to using this model. 

The  models mentioned above are explained
in Section~\ref{sec:prelim}. 

Combining the results obtained for degree sequences we obtain
asymptotic formulae for the number of  \tc{} $(n, m)$-graphs for the following three cases: $m/n
\to 1$, $m/n$ bounded away from $1$ and $m/n \to \infty$. We then
combine all three cases into a single formula
(Theorem~\ref{thm:main}). The pieces of the proof of this are finally gathered together in Section~\ref{s:together}. The final section adapts the method to counting 2-edge-connected graphs.

\section{Main results}

We assume that $m>n$. Let $\tcnm$ denote the number of labelled \tc{}
 (simple) graphs with $n$ vertices and $m$ edges. We may
assume that the vertex set is $[n]$.

In preparation for the statement of our results we define the odd
falling factorial $(2m-1)!! \eqdef (2m-1)(2m-3)\cdots 1$, and the
average degree  $c \eqdef
2m/n$.

Define $g:\setR_{++} \to \setR$ by $g(\lambda) \eqdef
\lambda(e^{\lambda}-1)/(e^{\lambda}-1-\lambda)$. Then $g$ is an
increasing function with $g(\lambda) \to 2$ as $\lambda\to 0$. Since
$c>2$, we may let $\lambda_c$ be the (unique) positive root of
\begin{equation*}
  g(\lambda) = \frac{\lambda(e^{\lambda}-1)}{e^{\lambda}-1-\lambda} = c,
\end{equation*}
and we set
\begin{equation*}
  \bar{\eta}_c\eqdef \frac{\lambda_c e^{\lambda_c}}{e^{\lambda_c}-1}
  \quad\text{and}\quad
  p_c\eqdef\frac{\lambda_c^2}{2\left(e^{\lambda_c}-1-\lambda_c\right)}.
\end{equation*}

Our main result is the following.
\begin{thm}
  \lab{thm:main}
  Suppose $m = O(n\log n)$ and $ m- n \to \infty$. Then
  \[ \tcnm \sim (2m-1)!!
  \frac{(\exp(\lambda_c)-1-\lambda_c)^n}
  {\lambda_c^{2m}\sqrt{2\pi n c(1 + \bar{\eta}_c-c)}} 
  \sqrt{\frac{c-2p_c}{c}} 
  \exp\left(-\frac{c}{2} -\frac{\lambda_c^2}{4} \right).
  \]
\end{thm}

To prove this, we first obtain asymptotic formulae for $\tcnm$ for the
following three cases: $c \to 2$, bounded $c > 2$, and $c\to \infty$.
\begin{thm}
  \lab{thm:maincases}
  Suppose $m = O(n\log n)$ and $r\eqdef 2m-2n \to \infty$. Then
  \begin{enumerate}
  \item[(a)] if $c \to 2$,
    \[ \tcnm \sim (2m-1)!!
    \frac{(\exp(\lambda_c)-1-\lambda_c)^n}
    {\lambda_c^{2m}\sqrt{2\pi n (c-2)}} 
   \cdot \frac{\sqrt{3r}}{e \sqrt{ 2m}}.
    \]

  \item[(b)] if $c = O(1)$ and $c > C_0$ for a constant $C_0> 2$ for $n$
    large enough, 
    \[ \tcnm \sim (2m-1)!!
    \frac{(\exp(\lambda_c)-1-\lambda_c)^n} {\lambda_c^{2m}\sqrt{2\pi n
        c(1 + \bar{\eta}_c-c)}} \sqrt{\frac{c-2p_c}{c}} \exp\left(
      -\frac{c}{2} - \frac{\lambda_c^2}{4} \right).
    \]
    
  \item[(c)] if $c \to \infty$,
    \[ \tcnm \sim (2m-1)!! 
    \frac{(\exp(\lambda_c)-1-\lambda_c)^n}
    {\lambda_c^{2m}\sqrt{2\pi n c}} 
    \exp\left( -\frac{\bar{\eta}_c}{2}
      -\frac{\bar{\eta}_c^2}{4} \right).
    \]
  \end{enumerate}
\end{thm}

For each case, we prove that the formula obtained is asymptotically
equivalent to the formula in Theorem~\ref{thm:main}. Then we show how
to combine these to obtain Theorem~\ref{thm:main}.

As we have already mentioned, Wright~\cite{Wright83} proved an
asymptotic formula for the case $k \eqdef m-n = o(\sqrt{n})$ with
$k\to\infty$. His formula is
\begin{equation*}
  \tcnm
  =
  a \sqrt{6\pi}  n^{n+3k-1/2} e^{2k-n} (18 k^2)^{-k} (1+ O(k^{-1})+O(k^2/n)),
\end{equation*}
where $a$ is a constant. Wright gave a method of estimating $a$, and
computed it to be $0.058549831\ldots$. Vobly{\u\i}~\cite{Vob87}
determined that $a$ is $1/(2e\pi)$.  To compare Wright's formula to
our own, we compute $\lambda_c =3 r/n - \frac{3}{2}(r/n)^2
+\frac{6}{5}(r/n)^3 + O((r/n)^4)$ for $k=o(n^{2/3})$, and then from
Theorem~\ref{thm:maincases}(c) it is easy to obtain the following.
\begin{cor}\label{cor4}
  Suppose that $k = m- n =o(n^{2/3})$ and $k\to \infty$. Then
  \begin{equation*}
    \tcnm
    \sim
    \frac{\sqrt{3}}{e \sqrt{ 2\pi}}
    n^{n+3k-1/2} e^{2k-n+3k^2/(2n)}(18k^2)^{-k}.
  \end{equation*}
\end{cor}
One can also check easily that for $m\approx n\log n$, our formula is asymptotic to the total number of $(n,m)$-graphs, in accordance with the result of Erd{\H o}s and R{\'e}nyi mentioned above.

 The proof of each case
in Theorem~\ref{thm:maincases} follows the same strategy. First we
study the general ``typical'' degree sequences of each case, computing
the (asymptotic) probability that a graph with a given degree sequence
is \tc. With this, we obtain an asymptotic formula for the number of
$2$-connected graphs with degree sequence $\ds$.  For all typical
sequences, we obtain the same probability (within a uniform
error). This allows us to sum over degree sequences obtaining an
asymptotic formula for {\tcnm} in each case.

We use $\Dcal(n,m)$ to represent the set of degree sequences
$\ds\eqdef (d_1,\dotsc, d_n)$ such that $\sum_{i=1}^n d_i = 2m$ and
$d_i \geq 2$ for all $i\in [n]$.    For
$\ds\in\Dcal(n,m)$ define $\eta(\ds) \eqdef \frac{1}{2m}\sum
d_j(d_j-1)$ and let $\tcd$ represent the number of \tc{} graphs with
degree sequence $\ds$. For every integer $j$, let $D_j = D_j(\ds)$
denote $\card{\set{i\st d_i = j}}$.

Define $\Ys=(Y_1,\dotsc,Y_n)$ to be a vector of independent 
\emph{truncated Poisson} random variables $Y\sim \tpoisson{2}{\lambda_c}$
defined by
\[
\prob{Y=j} =  \frac{\lambda_c^j}{j!( e^{\lambda_c} - 1 - \lambda_c )}
\]
for $j=2, 3, \dotsc$.

\begin{thm}
  \lab{thm:maindegrees}
  Suppose $m = O(n\log n)$ and $r\eqdef 2m-2n \to \infty$. Then
  \begin{enumerate}
  \item[(a)] Suppose further that $c \to 2$. Let $\psi(n) =
    r^{1-\eps}$ for some $\eps \in (0, 1/4)$. If $\ds \in \Dcal(n,m)$
    satisfies
    \begin{enumerate}
    \item[(i)] $|D_2 - \esp{D_2(\Ys)}| \leq \psi(n)$,
    \item[(ii)] $|D_3 -\esp{D_3(\Ys)}|\leq \psi(n)$,
    \item[(iii)] $|\sum_i \binom{d_i}{2} - \esp{\sum_i \binom{Y_i}{2}}| 
      \leq \psi(n)$,  
     \item[(iv)] $d_i \le 8\log(n-D_2(\ds))$ for every $i$,
    \end{enumerate}
    then
    \[ \tcd \sim  \frac{\sqrt{3r}}{e \sqrt{ 2m}}\cdot \frac{(2m-1)!!}{\prod_{j=1}^n d_j!}.
    \]
  \item[(b)] Suppose further that $c = O(1)$ and $c = c(n) = 2m/n >
    C_0$ for a constant $C_0> 2$ for $n$ large enough.  Let $\psi(n) =
    1/n^\eps$ for some $\eps \in (0,1/4)$. If $\ds \in \Dcal(n,m)$
    satisfies
    \begin{enumerate}
    \item[(i)] $d_i \le 6\log n$ for every $i$,
    \item[(ii)] $|\eta(\ds) - \bar \eta_c|\leq \psi(n)$ and
    \item[(iii)] $|D_2/n -p_c| \leq
      \psi(n)$, 
    \end{enumerate}
    then
    \[ \tcd \sim \frac{(2m-1)!!}{\prod_{j=1}^n d_j!}
    \sqrt{\frac{c-2p_c}{c}} \exp\left( -\frac{c}{2}
      -\frac{\lambda_c^2}{4} \right).
    \]
    
  \item[(c)] If $c \to \infty$ and $\ds \in \Dcal(n,m)$ satisfies $\max d_i
    \le n^\epsilon$ for some $\epsilon \in (0,0.01)$ then
    \[ \tcd \sim \frac{(2m-1)!!}{\prod_{j=1}^n d_j!} 
    \exp\left( -\frac{\eta(\ds)}{2}
      -\frac{\eta(\ds)^2}{4} \right).
    \]
  \end{enumerate}
\end{thm}


\section{Background and preliminary results}
\lab{sec:prelim}
\subsection{Models for graphs of given degree sequence}\lab{ss:models}
The \emph{pairing model} or {\em configuration model} is a standard
theoretical tool for studying graphs of a given degree sequence.  For
$\ds \in \Dcal(n,m)$ a random perfect matching is placed on a set of
$2m$ points which are grouped into $n$ cells of size $d_1, d_2,
\dotsc, d_n$.  This random pairing naturally corresponds in an obvious
way to a random pseudograph (possibly containing loops or parallel
edges)  of
degree sequence $\ds$ in which each cell becomes a vertex.

Let $U(\ds)$ denote the probability the random pairing model is
simple, and $U'(\ds)$ the probability that a random pairing is both
\tc{} and simple.  It is well known that the number of pairings
corresponding to a given (simple) graph is $\prod_{j=1}^n d_j!$, thus
\bel{eqn:pairingd} \tcd = \frac{(2m-1)!!}{\prod_{j=1}^n d_j!}U'(\ds).
\ee

Let $\cnm$ denote the number of labelled  (simple) graphs with
$n$ vertices and $m$ edges (with vertex set $[n]$) with minimum degree
at least $2$ and let
\begin{equation*}
Q(n,m) 
\eqdef 
\sum_{\ds\in\Dcal(n,m)}\prod_{j=1}^n
\frac{1}{d_j!}.  
\end{equation*}
Recall the definition of the sequence $\Ys = (Y_1,\ldots, Y_n)$ of
independent truncated Poisson random variables with
parameter~$(2,\lambda_c)$. Let $\Sigma$ denote the event that $\sum_i Y_i
= 2m$. Then~\cite[Eq.\ (13)]{PWa} states that \bel{pairingsimple} \cnm
= (2m-1)!!Q(n,m)\esp[\big]{U(\Ys)| \Sigma}.  \ee This was obtained by
summing the analogue of~\eqn{eqn:pairingd} for $U(\ds)$ over
$\ds\in\Dcal(n,m)$. Applying the same argument to~\eqn{eqn:pairingd},
one easily obtains
\bel{eqn:pairingm} \tcnm = (2m-1)!!Q(n,m)\esp[\big]{U'(\Ys)| \Sigma}.
\ee

This distribution $\Ys$ has been studied by several authors. Facts
about it will be introduced as they are needed. The probability of the
event $\Sigma$ has been estimated quite precisely in terms of the
variance of $Y$. (See Lemma~2 and Theorem~4(a) of \cite{PWa}.)

We will use the estimate in Theorem~4(a) of \cite{PWa}:
\bel{eqn:probSigma} \prob{\Sigma} = \frac{1 + O(r^{-1})}{\sqrt{2\pi n
    c (1 + \bar\eta_c - c)}} =\Omega(1/\sqrt{r}), \ee where $r:=
2m-2n$ under the conditions that $r=O(n \log
n)$ and $r\to\infty$.  One of the reasons for which~\eqn{eqn:pairingm}
holds is that $Q(n,m)$ can be rewritten as $ (e^{\lambda_c} - 1 -
\lambda_c)^n \lambda_c^{-2m} \prob{\Sigma}$, and
so~\eqn{eqn:probSigma} gives
\begin{equation}
  \lab{eqn:Qnmdef}
  Q(n,m)
  =
  \frac{(e^{\lambda_c} - 1 - \lambda_c)^n}
  {\lambda_c^{2m}\sqrt{2\pi n c (1 + \bar\eta_c - c)}}
  (1 + O(r^{-1})).
\end{equation}

When the degrees are all at least 2, the \emph{kernel configuration
  model} of Pittel and Wormald \cite{PWb} can provide some advantages.
Before describing the model we need some definitions. The
\emph{2-core} of a graph is its maximal subgraph of minimum degree at
least 2.  The \emph{pre-kernel} of a graph is obtained from the
$2$-core by throwing away any components which are simply cycles. The
\emph{kernel} of a graph is obtained from the pre-kernel by replacing
each maximal path of degree-2 vertices by a single edge.  We say that
a pseudograph is a pre-kernel (respectively, a kernel) if it is the
pre-kernel (respectively, kernel) of some graph.  Now we are ready to
describe the kernel configuration model for a degree sequence $\ds \in
\Dcal(n,m)$.

For each $i$ with $d_i \geq 3$ create a set $S_i$ of $d_i$ points.
Choose, uniformly at random, a perfect matching on the union of these
sets of points. Assign the remaining numbers $\{ i : d_i=2 \}$ to the
edges of the perfect matching and, for each edge, choose a linear
order for these numbers.  The assignment and the linear ordering are
chosen uniformly at random. The pairing and assignment (with linear
orderings) are the configuration. A pseudograph $G$ is constructed by
collapsing each set $S_i$ to a vertex (producing a kernel $K$) and
placing the degree-2 vertices on the edges of the kernel according to
the assignment and linear orderings.

It is not hard to show (see Corollary 2 in \cite{PWb}) that each
pre-kernel can be produced by the same number of configurations, and
\[
\tcd
=
\frac{(2m'-1)!!(m-1)!\prob{\tcs(\ds)}}{(m'-1)!\prod_{i\in R(\ds)} d_i!},
\]
where $R = R(\ds) \eqdef \set{i\in [n] \st d_i \ge 3}$, $m' = m'(\ds)
\eqdef \frac12 \sum_{i\in R}d_i$, and $ \tcs(\ds)$ is the event that
the pre-kernel produced by the kernel configuration model is \tc{} and
simple. For later use, let $n' = n'(\ds)\eqdef \card{R(\ds)} =
\sum_{j\geq 3} D_j(\ds)$. By Stirling's formula, \bel{eqn:kcmd} \tcd =
\frac{(2m-1)!!\sqrt{m'/m}\prob{\tcs}}{\prod_{i=1}^n d_i!}  \left( 1 -
  O(1/m')\right).  \ee In~\cite[(5.3)]{PWb}, a similar expression for
the event of being connected and simple was summed over
$\ds\in\Dcal(n,m)$.  We can use the same argument,
using~\eqn{eqn:probSigma} and~\eqn{eqn:Qnmdef}, and defining
\bel{wdef} w(\ds) \eqdef \prob{\tcs(\ds)}\sqrt{m'}, \ee to get
\begin{equation}
  \begin{split}
    \lab{eqn:kcmm} 
    \tcnm &=  
    \left( 1 - O(1/m')\right)
    (2m-1)!!Q(n,m)\sqrt{m^{-1}}\esp{w(\Ys)|\Sigma}. 
    \\
    &=
    \left( 1 - O(1/r)\right)
    (2m-1)!!
    \frac{(e^{\lambda_c} - 1 - \lambda_c)^n}
    {\lambda_c^{2m}\sqrt{2\pi n c (1 + \bar\eta_c - c)}}
    \sqrt{m^{-1}}\esp{w(\Ys)|\Sigma}.   
  \end{split}
\end{equation}

\subsection{Relation between vertex and edge connectivity}
In this section we investigate some properties of \tc{} graphs which
may be of independent interest. We show that, asymptotically almost
surely, a random kernel is \tc{} if and only if it is \tec. (An event
is said to occur asymptotically almost surely (\aas) if its
probability is $1-o(1)$.) In this article, for convenience, we allow
\tc{} pseudographs and \tec{} pseudographs to have loops. In
particular, a cut-vertex of a pseudograph is a vertex whose removal
(along with all incident edges) increases the number of components,
and a graph is 2-connected if it has no cutvertices and at least three
vertices.

\begin{prop}
  \lab{prop:kernelequiv} Let $\ds\in\Dcal(n,m)$ satisfying $n\geq 3$
  and $3 \leq \delta = d_1\leq \cdots \leq d_n = \Delta \leq
  n^{0.04}$. Let $K$ be the kernel of the random pseudograph produced
  by the pairing model using degree sequence $\ds$.  \Aas, $K$ is
  \tc{} iff it is \tec.
\end{prop}
\begin{proof}
  Let $K$ be the random kernel produced by the pairing model using
  degree sequence $\ds$ satisfying $n\geq 3$ and $3 \leq \delta =
  d_1\leq \cdots \leq d_n = \Delta \leq n^{0.04}$.
  By closely following \Luczak's proofs of properties of
  (simple) graphs with given degree sequence in~\cite[Section~12.3]{Luczak}),
  it is straightforward to prove the following lemmas.
  \begin{lem}
    \label{lem:smallsets}
    \Aas, no subgraph of $K$ with $s$ vertices, $2\leq s\leq n^{0.4}$,
    has more than $1.2s$ edges.
  \end{lem}
  \begin{lem}
    \label{lem:largesets}
    \Aas, each subset of $K$ with $s$
    vertices, $n^{0.3}\leq s\leq \ceil{n/2}$ has more than $\delta$
    neighbours.
  \end{lem}
  Suppose that $v$ is a cut-vertex in $K$ not in a bridge. Then $v$
  decomposes $K$ into components $W_1$ and $W_2$ with $\card{W_1}\leq
  \card{W_2}$. Note that $v$ sends at least $2$ edges to $W_1$ and at
  least $2$ edges to $W_2$. (Otherwise $v$ would be in a bridge).

  Suppose that $\card{W_1}=1$. Then the number of edges induced by
  $W_1\cup\set{v}$ is at least $3$ (since $\delta\geq 3$) which
  is $ \frac{3}{2} \card{W_1\cup \set{v}}$.
  On the other hand, if $\card{W_1}\geq 2$, the number of edges
  induced by $W_1\cup\set{v}$ is at least $(3\card{W_1}+2)/2\ge 1.25
  \card{W_1\cup\set{v}}.$ For $ \card{W_1\cup\set{v}}\le n^{0.4}$, we
  conclude that such $v$ \aas{} does not exist, by
  Lemma~\ref{lem:smallsets}.  Otherwise, $ \card{W_1}\geq n^{0.3}$ and such $v$ \aas{}
  does not exist by Lemma~\ref{lem:largesets}.
  
  So \aas{} $K$ has a bridge if it has a cut-vertex. The converse is
  deterministically true for pseudographs with at least three vertices,
  and the proposition follows.
\end{proof}

Note that Lemmas~\ref{lem:smallsets} and~\ref{lem:largesets} actually
imply that \aas{} there are no cut-sets of cardinality from $2$ to
$\delta-1$ inclusive.


\section{The case $c \to 2$}

In this case, we can directly implement the plan presented in the
introduction: we examine the probability that a random $n$-vertex
graph is 2-connected when its vertex degrees are chosen as independent
truncated Poissons random variables, conditioned on the sum being
$2m$. We do this for typical degree sequences and then transfer this
result to a random $(n,m)$-graph with minimum degree 2.

Recall the definition of the sequence $\Ys$ of independent truncated
Poisson random variables and the associated event $\Sigma$ 
used in~\eqn{eqn:kcmm}. 

Define
$\mu_2 = \esp{D_2(\Ys)}$, $\mu_3 = \esp{D_3(\Ys)}$ and $\mu =
\esp{\sum_{i=1}^n \binom{Y_i}{2}}$.  We need to know the asymptotic
behaviour of these expected values.

\begin{lem}
  \lab{lem:expected1b}
  We have  
 $\mu_2 = n-r+o(r)$, 
  $\mu_3 = r+ o(r)$
  and
 $ \mu = n+2r+o(r)$.
\end{lem}
The proof of this lemma is straightforward and depends only on
properties of $\Ys$ that follow easily from facts established by
Pittel and Wormald \cite{PWa, PWb}. The proof is presented at the end
of the section.

We now define a set of ``typical'' degree sequences. Let $\psi(n):
\setN \to \setR_+$ be any function such that $\psi(n) = o(r)$. (We
will specify a particular such function later.) Recall the
definition $n'(\ds) = \sum_{j\geq 3} D_j$. Let
\begin{equation*}
  \begin{split}
    \tilde\Dcal(\psi) \eqdef
    \bigg\{
    \ds \in \Dcal(n,m)\st
    &|D_2(\ds)-\mu_2|\leq \psi(n);
    \ |D_3(\ds)-\mu_3|\leq \psi(n);
    \\
    & \Big|\sum_{i=1}^{n}\binom{d_i}{2} -\mu\Big|\leq \psi(n);
    \ \max_i d_i \leq 8\log n'(\ds) 
    \bigg\}.
  \end{split}
\end{equation*}
and define $\tilde\Dcal^{c}(\psi) \eqdef \Dcal(n,m)\setminus
\tilde\Dcal(\psi)$.

Let $\ds\in\tilde\Dcal(\psi)$. We want to compute the probability of
$\tcs(\ds)$ as in~\eqn{wdef}. It easy to see that this is the same as
the event that $G$ is simple and $K$ is $2$-connected and loopless
(but permitting $K$ to have multiple edges). Let $B$ denote the event
that $G$ is simple and $K$ is \tec{} and has no loops.  The maximum
degree in $K$ is at most $8\log(n') < (n')^{0.04}$ so we may use
Proposition~\ref{prop:kernelequiv} to deduce $\prob{B} =
\prob{\tcs(\ds)}+o(1)$.

We have that $\max d_i \le 6\log n$ and, by Lemma~\ref{lem:expected1b}
and the definition of $\tilde\Dcal(\psi)$,
\begin{equation*}
  \sum_{i\in R(\ds)}
  \binom{d_i}{2}
  =
  \sum_{i=1}^n
  \binom{d_i}{2}
  -D_2(\ds)
  =
  n + 2r +o(r)
  - (n - r +o(r))
  =
  3r+o(r)
  < 4r,
\end{equation*}
for large $n$, which are sufficient conditions (by Lemma~5 in
\cite{PWb}) for a random kernel configuration to be \aas\
simple. Thus, the probability of $G$ being simple is $1+o(1)$.
(Actually, in~\cite{PWb} it is shown to be $1-O(r/n+1/r)$.)

For a random pairing having a given degree sequence of minimum degree
at least 3, the probability of being 2-edge-connectivity was
investigated by \Luczak{} in \cite{Luczak}.  He shows (in his Lemma
12.1(iii)) that this probability approaches $\exp\big(-\frac{3}{2} D_3/m'
\big)$ provided $D_3/m'$ approaches a positive constant. Using
Lemma~\ref{lem:expected1b} for $\mu_2$,
\begin{align*}
  m'(\ds) &= m - D_2(\ds) = m - \mu_2 + O(\psi) = m - \mu_2 + o(r)
 = \frac{3r+o(r)}{2}.
\end{align*}
Applying this to $K$, we have
\begin{equation*}
  \frac{D_3(\ds)}{m'}
  =
  \frac{r+o(r)}{(3/2)r+o(r)}
  \sim
  \frac{2}{3}
\end{equation*}
so the probability that $K$ is \tec{} goes to $1/e$.  Note that $K$
being \tec{} implies that there are no loops on vertices of degree $3$
in $K$.  The expected number of loops in $K$ on vertices of degree at
least 4 is
\[
\sum_{i:d_i\ge 4}\binom{d_i}{2}\frac{1}{2m'-1}
=
\left(\sum_{i=1}^n\binom{d_i}{2}-D_2-3D_3\right)\frac{1}{3r+o(r)}
=o(1)
\]
by Lemma~\ref{lem:expected1b}, so \aas{} no such loops exist.
We conclude that 
  $\prob{B} \sim 1/e$, and thus 
\begin{equation}
  \lab{eqn:probcase1b}
  \prob{\tcs(\ds)}\sim \frac{1}{e}.
\end{equation}

These two results together with~\eqn{eqn:kcmd}
give Theorem~\ref{thm:maindegrees}(a) and, in particular recalling
$w(\ds) = \prob{\tcs(\ds)}\sqrt{m'}$ from~\eqn{wdef},
\[
w(\ds) 
\sim \frac{1}{e}\sqrt{\frac{3r}{2}}.
\]

Finally, we will show Theorem~\ref{thm:maincases}(a). Let $\psi(n) =
r^{1-\eps}$ for some $\eps \in (0, 1/4)$. We will show
\begin{equation}
  \lab{eqn:typicalcase1bprob}
  \prob{\tilde\Dcal(\psi) | \Sigma} = 1 + O(\sqrt{r}/n) + O(r^{2\eps-1/2}).
\end{equation}
Then using the formula for $w(\ds)$ shown above for any
$\ds\in\tilde\Dcal(\psi)$, we have
\begin{equation*}
  \begin{split}
    \esp{w(\Ys)|\Sigma} &= 
    \esp{w(\Ys)|\tilde\Dcal(\psi)}\prob{\tilde\Dcal(\psi) | \Sigma}
   + \esp{w(\Ys)|\tilde\Dcal^c(\psi)}
    \prob{\tilde\Dcal^c(\psi) | \Sigma}
    \\
    &=
    \esp{w(\Ys)|\tilde\Dcal(\psi)}(1-O(\sqrt{r}/n)-O(r^{2\eps-1/2})) +o(\sqrt{r})
    \\
    &\sim \frac{1}{e}\sqrt{\frac{3r}{2}},
  \end{split}
\end{equation*}
which combined with~\eqn{eqn:kcmm}, and using 
\bel{csmall}
c (1 + \bar\eta_c - c) \sim c-2
\ee
for $c\to 2$ (see~\cite{PWa}(20)), gives the conclusion of
Theorem~\ref{thm:maincases}(a).
 
So it suffices to prove~\eqn{eqn:typicalcase1bprob}. Let $p_i$ denote
the probability that a variable with distribution
$\tpoisson{2}{\lambda_c}$ has value $i$. Recall that $\eta(\ds) =
(\sum_{i=1}^{n} d_i(d_i-1))/(\sum_{i=1}^{n} d_i)$. First we will study
the first three conditions in the definition of
$\tilde\Dcal(\psi)$. Let $F$ be the event that $\Ys$ fails to satisfy
any of the three conditions. Using Chebyshev's inequality, together
with $ p_2(1-p_2)n = O(r)$ and $ p_3(1-p_3)n \leq p_3 n = O(r) $ by
straightforward calculations, we have that
\begin{equation*}
  \prob{|D_2(\Ys) - \mu_2|\geq \psi(n)}
  \leq
  \frac{p_2(1-p_2)n}{\psi(n)^2} 
  = O\left(\frac{r}{\psi(n)^2}\right)
\end{equation*}
and 
\begin{equation*}
  \prob{|D_3(\Ys) - \mu_3|\geq \psi(n)}
  \leq
  \frac{p_3(1-p_3)n}{\psi(n)^2}
  = O\left(\frac{r}{\psi(n)^2}\right).
\end{equation*}
There is a concentration result for $\mu$ shown in~\cite[p.  262]{PWa}  
which may be expressed as 
\begin{equation*}
  \prob[\bigg]{\Big|\sum_{i=1}^n\binom{Y_i}{2} - \mu\Big|\geq \psi(n)}
  =
  O\left(\frac{\lambda_c m^2}{n\psi(n)^2}\right).
\end{equation*}
In \cite{PWa},
Lemma~1(a) states \bel{eqn:lambda1b} \lambda_c = 3(c-2)+O((c-2)^2) =
3r/n + O(r^2/n^2).  \ee Thus,
\begin{equation*}
  \begin{split}
    \frac{\lambda_c m^2}{n}
    =
    \frac{3r m^2}{n^2} +
    O
    \left(
      \frac{r^2 m^2}{n^3}
    \right)
    =
    O(r).
  \end{split}
\end{equation*}
Hence
\begin{equation*}
  \prob[\bigg]{\Big|\sum_{i=1}^n\binom{Y_i}{2} - \mu\Big|\geq \psi(n)}
  =
  O\left(\frac{r}{\psi(n)^2}\right).
\end{equation*}
This implies that $\prob{F} = O( r/\psi(n)^2)$.
Using~\eqn{eqn:probSigma} we get
\begin{equation*}
  \prob{F | \Sigma}
  \leq
  \frac{\prob{F}}{\prob{\Sigma}}
  =
  O(\sqrt{r})O\left(\frac{r}{\psi(n)^2} \right)
  =
  O\left(\frac{r^{3/2}}{\psi(n)^2} \right).
\end{equation*}

Now consider the last condition in the
definition of $\tilde\Dcal(\psi)$: $\max_i d_i \leq 8\log n'(\ds)$. If
the first condition in the definition of $\tilde \Dcal(\psi)$ holds,
then, using Lemma~\ref{lem:expected1b}, we have $D_2(\ds) = n - r +
\phi(n)$ for some function $\phi(n) = o(r)$ and so $n'(\ds) = r -
\phi(n)$. Let $F'$ denote the event that the first condition holds but
the last condition fails. Thus, $\prob{F'} \leq \prob[\big]{\max_i Y_i \geq
  8\log(r - \phi(n))}$. For $r \leq \sqrt{n}$, it is easy to see that
$\esp{D_j(\Ys)} = O\left( r^{j-2}/n^{j-3} \right)$ for every $j\geq
4$. Thus, using Markov's inequality and the union bound, one can prove
that
\begin{equation*}
  \prob{D_j(\Ys) \geq 1 \text{ for some }j\geq 8}\leq n\cdot O(1/n^2)
  =
  O(1/n).
\end{equation*}
For $r > \sqrt{n}$, it is easy to bound the tail probability of $Y_i$
(see (3.17) of~\cite{PWb} for example) as
\begin{equation*}
  \prob{Y_i \geq 8\log (r-\phi(n))}
  =
  O\left(
    \exp(-4\log (r-\phi(n)))
  \right)
  =
  O\left(
    \exp(-4\log r  )
  \right)
  =
  O\left(
    \frac{1}{n^2}
  \right).
\end{equation*}
Thus, $\prob{\max_i Y_i > 8\log (r-\phi(n))} = O(1/n)$. Since
$\prob{\Sigma} = \Omega(1/\sqrt{r})$, we conclude that
\begin{equation*}
  \prob{F' | \Sigma}
  \leq
  O(\sqrt{r}) O(1/n)
  =
  O(\sqrt{r}/n).
\end{equation*}
Hence
\begin{equation*}
  \prob{\tilde\Dcal(\psi) | \Sigma}
  \geq
  1 - \prob{F |\Sigma} - \prob{F' | \Sigma}
  =
  1 +  O\left(\frac{r\sqrt{r}}{\psi(n)^2} \right)
  + O(\sqrt{r}/n) 
  =
  1 + O(r^{2\eps-1/2})  + O(\sqrt{r}/n),
\end{equation*}
and we proved~\eqn{eqn:typicalcase1bprob}.

\begin{proof}[Proof of Lemma~\ref{lem:expected1b}]
  Let $r(\Ys)\eqdef \sum_{i=1}^n Y_i - 2n$ and
  $n'(\Ys)\eqdef n - D_2(\Ys)$. Note that $r(\Ys)$ may not
  coincide with $r$ because we are not conditioning on $\Sigma$. But
  \begin{equation}
    \label{eq:esp_rY}
    \esp{r(\Ys)}
    =
    \esp[\Big]{\sum_{i=1}^n Y_i} - 2n
    =
    \sum_{i=1}^n  \frac{2m}{n} - 2n
    =
    2m-2n
    =r.
  \end{equation}

  Note that 
  \begin{equation*}
    \sum_{i=1}^n Y_i 
    = 
    \sum_{i\in R(\Ys)} Y_i +  2D_2(\Ys)
    \geq
    \sum_{i\in R(\Ys)} 3  + 2n - 2n'(\Ys)
    =
    3n'(\Ys)  + 2n - 2n'(\Ys)
    =
    n'(\Ys)+2n.
  \end{equation*}
  Hence, $n'(\Ys)\leq r(\Ys)$.

  Thus, 
  \begin{equation}
    \label{eq:lower_D_2}
    D_2(\Ys) = n-n'(\Ys) \geq n-r(\Ys).
  \end{equation}
  and so, by~\eqref{eq:esp_rY},
  \begin{equation}
    \label{eq:lower_esp_D_2}
    \esp{D_2(\Ys)} \geq n- r.
  \end{equation}
  Moreover, $D_2(\Ys) \leq n - D_3(\Ys)$, which implies that
  \begin{equation}
    \label{eq:upper_D_2}
    \esp{D_2(\Ys)}\leq n -\esp{D_3(\Ys)}.
  \end{equation}

  Since $n-r= n+o(n)$ and $n -D_3(\Ys)\leq n$, we conclude that
  $\esp{D_2(\Ys)} = n+o(n)$. 
  Using~\eqn{eqn:lambda1b},
  \begin{equation*}
    \begin{split}
      \esp{D_3(\Ys)}
      &=
      \frac{\lambda_c^3}{3! (e^{\lambda_c}-1-\lambda_c)}n
      =
      \frac{\lambda_c}{3}\esp{D_2(\Ys)}
      =
      \left(
        \frac{r}{n}
        +
        O\left(
          \frac{r^2}{n^2}
        \right)
      \right)
      (n+o(n))
      \\
      &=
      r+o(r)+O(r^2/n)
      =
      r+o(r).
    \end{split}
  \end{equation*}

  By~\eqref{eq:upper_D_2}, $\esp{D_2(\Ys)}\leq n-\esp{D_3(\Ys)} =
  n-r+o(r)$. So by~\eqref{eq:lower_esp_D_2}, we conclude that
  $\esp{D_2(\Ys)} = n-r+o(r)$.

  In \cite{PWb}, the line after (5.6) (with error term corrected to $O(r^2/n)$) 
  states
  $\esp[\Big]{\sum_{i=1}^n\binom{Y_i}{2}} = n+2r+O(r^2/n) = n+2r+o(r)$.
\end{proof}


\section{The case $c$ bounded away from $2$, and bounded}
\lab{sec:case1a}
Let $\psi:\setN\to \setR$ be a function such that $\psi(n) = o(1)$.
Let
\begin{equation*}
  {\tilde\Dcal}(\psi)
  \eqdef
  \set[\Big]{\ds\in\Dcal(n,m)\st
    d_i\leq 6\log n \ \forall i;\ 
    |\eta(\ds)- \bar{\eta_c}| \leq \psi(n);\
    |D_2(\ds) - p_c n|\leq n\psi(n) 
  }
\end{equation*}
Let ${\tilde\Dcal^c(\psi)} \eqdef \set{\ds\in \setN^n\st d_i\geq 2 \
  \forall i; \ds\not\in\tilde\Dcal(\psi)}$. (Note that if
$\ds\in\tilde\Dcal(\psi)$ then $\sum d_i =2m$ but we do not have this
constraint for ${\tilde\Dcal^c(\psi)}$.)

Let $\ds\in \tilde\Dcal(\psi)$.  We use the kernel configuration model
to investigate the graphs with no isolated cycles and with degree
sequences in~$\tilde\Dcal(\psi)$. According to the general plan in the
introduction, we will then see that the probability such graphs are
2-connected is concentrated around a given value when the degree
sequence consists of independent truncated Poissons, and show how this
probability then carries over to random graphs with a given number of
edges.

Let $\ds'$ be the restriction of $\ds$ to the coordinates with value
at least $3$, and let $G$ be obtained using the kernel configuration
model with degree sequence $\ds$. Recall $n'= \card{\set{i\st d_i\geq
    3}}$.

Let $P$ be the random perfect matching placed on a set
$S$ with $\sum_{i=1}^{n'} d_i'$ points grouped in cells of size $d_1',
d_2',\ldots, d_{n'}'$. Let $K$ be the kernel (obtained by contracting
the cells of $P$). Let $v_i$ denote the vertex with degree $d_i'$
in~$K$. Let $M$ denote the number of edges in $K$.

We want to compute the probability that $G$ is $2$-connected and
simple.  Let $B$ be the event that $G$ is simple and that $K$ is
\tec{} and has no loops. Since $n' = (1-p_c)n + o(n) = \Theta(n)$, we
have $\max_i d_i \le 6 \log n \le (n')^{0.04}$, and so
Proposition~\ref{prop:kernelequiv} says that, conditioning on $B$, $K$
is \aas{} \tc{}. If $K$ is \tc{} and loopless, it is easy to show that
$G$ is also \tc.  In other words,
\begin{equation*}
  \prob{\tcs|B}
  = (1+o(1)).
\end{equation*}
Note that $\tcs \subseteq B$ for $n > 2$.

Let $A$ denote the event that $G$ has no multiple edges and $K$ has no
loops.  \Luczak{} has shown (see Lemma~12.1(ii) in \cite{Luczak}) that in
a random pseudograph with given degree sequence, with the distribution
of pairing model, having minimum degree at least 3, \aas{} all \tec{}
components, except at most one, are loops at vertices of degree
$3$. Hence, $ \prob{A\setminus B} = o(1)$. Since $B\subseteq A$, we
deduce $ \prob{A} = \prob{B} + o(1)$.

We will show that 
\begin{equation}
  \lab{eqn:probA}
  \prob{A} \sim p_a \eqdef \exp(-c/2 -\lambda_c^2/4).
\end{equation}
Hence,
\begin{equation}
  \lab{eqn:probtcscase1a}
  \prob{\tcs}
  =
  \prob{\tcs | B} \prob{B}
  =
  (1+o(1))\prob{A}
  \sim
  p_a.
\end{equation}

Note that $D_2(\ds) = p_cn + nO(\psi(n))\sim p_cn$ and $\eta(\ds) =
\bar\eta_c + O(\psi(n))\sim \bar\eta_c$.  Thus
\begin{equation*}
  \sqrt{\frac{m'(\ds)}{m}}
  =
  \sqrt{\frac{m-D_2(\ds)}{m}}
  =
  \sqrt{\frac{(c/2)n-p_c n + o(n)}{(c/2)n}}
   \sim
  \sqrt{\frac{c-2p_c}{c}}
\end{equation*}
since $c>2$ and $p_c\leq 1$. Using this fact together
with~\eqn{eqn:probtcscase1a},
\begin{equation}
  \lab{eqn:functionwcase1a}
 \prob{\tcs(\ds)}\sqrt{m'}
  \sim
  \sqrt{m}
  \sqrt{\frac{c-2p_c}{c}}p_a ,
\end{equation}
which together with~\eqn{eqn:kcmd} proves Theorem~\ref{thm:maindegrees}(b).

So in order to prove Theorem~\ref{thm:maindegrees}(b), it suffices to
prove~\eqn{eqn:probA}. The proof is presented in
Section~\ref{sec:probcase1a}.

We now prove Theorem~\ref{thm:maincases}(b). First we show that
\begin{equation}
  \lab{eqn:probtypical}
  \prob{\Ys\in {\tilde\Dcal^c(\psi)}} = O(n^{-1}\psi(n)^{-2})
  \quad\text{and}\quad
  \prob{\Ys\in {\tilde\Dcal^c(\psi)}| \Sigma} =
  O(n^{-1/2}\psi(n)^{-2}).
\end{equation}

We will use some properties of $\Ys$ developed by Pittel and Wormald
\cite{PWa}. Equation~(27) in \cite{PWa} states that $\prob{Y \ge j_0}
= O(\exp(-j_0/2))$ provided $j_0>2e\lambda_c$, where $Y\sim
\tpoisson{2}{\lambda_c}$.  Lemma~1(b) in the same paper assures
$\lambda_c\le 2m/n$, which is $O(1)$ in the present case, allowing us
to choose $j_0=6\log n$, apply the union bound, and conclude
\begin{equation*}
  \prob{\max_i Y_i > 6\log n}
  =O\left(\frac{1}{n^2}\right).
\end{equation*}
Note that $D_2(\Ys)$ has binomial distribution with probability
$p_c$. Using Chebyshev's
inequality,
\begin{equation*}
  \prob{|D_2(\Ys) - p_c n|\geq n\psi(n)}
  \leq
  \frac{p_c(1-p_c)n}{n^{2}\psi(n)^2}
  = O\left( \frac{1}{n\psi(n)^2}\right)
\end{equation*}
since $0\leq p_c\leq 1$

Pittel and Wormald also show (see~\cite[p. 262]{PWa}), 
\begin{equation*}
  \prob{|\eta(\Ys)-\bar\eta_c|\geq \alpha}
  =
  O\left( \frac{\lambda_c}{n\alpha^2}\right).
\end{equation*}
Since $\lambda_c\leq c = O(1)$,
\begin{equation*}
  \prob{|\eta(\Ys)-\bar\eta_c|\geq \psi(n)}
  =
  O\left( \frac{\lambda_c}{n\psi(n)^2}\right)
  =
  O\left( \frac{1}{n\psi(n)^2}\right).
\end{equation*}

Hence,
\begin{equation*}
  \prob{\Ys\in \tilde\Dcal^c}
  =
  O\left( \frac{1}{n\psi(n)^2}\right).
\end{equation*}

Since $r\eqdef 2m-2n=\Theta(n)$, \eqn{eqn:probSigma} implies that
$\prob{\Sigma} =
\Omega(1/\sqrt{n})$. 
Conditioning on $\Sigma$, we have
\begin{equation*}
  \prob{\Ys\in \tilde\Dcal^c| \Sigma}
  \leq
  \frac{\prob{\Ys\in \tilde\Dcal^c}}{\prob{\Sigma}}
  =
  O\left(
    \frac{n^{1/2}}{n\psi(n)^2}
  \right)
  =
  O\left(
    \frac{1}{n^{1/2}\psi(n)^2}
  \right).
\end{equation*}
This proves~\eqn{eqn:probtypical}.

Let $\psi(n)   =n^{-\eps}$ for some constant $\eps
\in (0,1/4)$.  Using~\eqn{eqn:functionwcase1a}
and~\eqn{eqn:probtypical},
\begin{equation*}
  \begin{split}
    \esp{w(\Ys) | \Sigma}
    &=
    \esp{w(\Ys)|\tilde\Dcal(\psi)}\prob{\tilde\Dcal(\psi) | \Sigma}
    +
    \esp{w(\Ys)|\Sigma\cap \tilde\Dcal^c(\psi)}\prob{\tilde\Dcal^c(\psi) | \Sigma}.
  \end{split}
\end{equation*}
Note that $w(\Ys)\leq \sqrt{m'}$ since $\prob{\tcs}\leq
1$. By~\eqn{eqn:probtypical}, we have that $\prob{\tilde\Dcal^c(\psi)
  | \Sigma} = O(1/n^{1/2-2\eps})$. So $\esp{w(\Ys)|\Sigma\cap
  \tilde\Dcal^c}\prob{\tilde\Dcal^c | \Sigma} =
O(\sqrt{m'}/n^{1/2-2\eps})$. Hence,
\begin{equation*}
  \begin{split}
    \esp{w(\Ys) | \Sigma} &= \esp{w(\Ys)|\Sigma\cap
      \tilde\Dcal(\psi)}(1-O(1/n^{1/2-2\eps})) +
    O(\sqrt{m'}/n^{1/2-2\eps})
    \\
    &= \sqrt{m} \sqrt{\frac{c-2p_c}{c}}p_a
    (1+o(1))(1-O(1/n^{1/2-2\eps})) + O(\sqrt{m'}/n^{1/2-2\eps})
    \\
    &= \sqrt{m} \sqrt{\frac{c-2p_c}{c}}\exp(-c/2 -\lambda_c^2/4)
    (1+o(1)),
  \end{split}
\end{equation*}
which together with~\eqn{eqn:kcmm} proves
Theorem~\ref{thm:maincases}(b).

\subsection{Showing $\prob{A}\sim p_a$}
\lab{sec:probcase1a}

Here we show~\eqn{eqn:probA}. Recall that $\ds\in \tilde\Dcal(\psi)$.
Let $e_1,\dotsc,e_\ell$ denote the possible loops in $K$. For every
$1\leq i\leq \ell$, let $X_i$ be the indicator variable for $e_i\in
E(K)$. Let $X= \sum_{i=1}^\ell X_i$. Let $f_1,\dotsc, f_t$ denote the
possible double edges in $K$ (here we do not include double loops).
For every $1\leq j\leq t$, let $Y_j$ be the indicator variable for
$f_j\subseteq E(G)$. Let $Y=\sum_{j=1}^t Y_j$.

Using the method of moments, we will show that
$X+Y\stackrel{\textrm{d}}{\to}
\poisson{c/2+ \lambda_c^2/4}$, which gives~\eqn{eqn:probA}. For this we need to show, for
every positive integer $k$, that
\begin{equation*}
  \esp{[X+Y]_k} = \left( \frac{c}{2}+\frac{\lambda_c^2}{4}\right)^k + o(1).
\end{equation*}

Considering the first moment, note that for every $1\leq i\leq
\ell$, we have that $\prob{X_i = 1}\sim 1/(2M)$. For the double
edges, we need to know the probability that a given set of edges of
the kernel is not assigned any vertices of degree 2 in the kernel
configuration model.  Let
\begin{equation}
  \delta=
    \left(\frac{c - 2p_c}{c}\right)^2
  = \left(\frac{\lambda_c}{c}\right)^2.
\end{equation}
For any fixed $q$ and any set of edges $\set{e_1,\ldots, e_q}$ in $K$, the probability that none of these kernel edges is assigned a vertex of degree 2 (and hence become edges of $G$) can be estimated as follows.
  \begin{equation}
    \begin{split}
      \prob{\set{e_1,\ldots,e_q}\subseteq E(G) | 
        \set{e_1,\ldots,e_q}\subseteq E(K)}
      &=
      \prod_{i=0}^{D_2-1} \left(1-\frac{q}{M+i}\right)
      \sim
      \exp\left(
        -q\sum_{i=0}^{D_2-1} \frac{1}{M+i}
      \right)
      \\
      &\sim
      \left(\frac{M-1}{M+D_2-1}\right)^q
      \sim
      \left(\frac{cn/2 - p_cn}{cn/2}\right)^q   
      =\delta^{q/2}.   \lab{delta}
    \end{split}
  \end{equation}
  Thus, for every $1\leq j\leq t$, we have that
  \begin{equation*}
    \prob{Y_j = 1}= 
    \prob{f_j\subseteq E(K)}
    \cdot\prob{f_j\subseteq E(G) | f_j\subseteq E(K)}
    \sim
    \frac{\delta}{(2M)^2}.
      \end{equation*}
  Hence,
  \begin{equation*}
    \begin{split}
      \esp{X+Y}
      &=
      \esp{X}+\esp{Y}
      \sim
      \ell\cdot\frac{1}{2M}
      +
      t\cdot\frac{\delta}{(2M)^2}
      \\  
      &=
      \frac{\sum_{i=1}^{n'}\binom{d_i'}{2}}{2M}
      +
      \frac{\delta}{(2M)^2}\sum_{\substack{(i,j)\\i\neq j}} 
      \binom{d_i'}{2}\binom{d_j'}{2}.
    \end{split}
  \end{equation*}
  We will use the following lemma, which is proved in the end of the
  section.
\begin{lem}
  \label{lem:fixed_k}
  Let $q$ be a fixed positive integer. For $d\in \tilde\Dcal(\psi)$,
  \begin{equation*}
    {\sum_{\substack{(i_1,\dotsc,i_q)}} 
      \prod_{j=1}^q \binom{d_{i_j}'}{2}}\cdot\frac{1}{(2M)^q}
    \sim
    \left(\frac{c}{2}\right)^q,
  \end{equation*}
  where the sum is over all $(i_1,\dotsc,i_q)\in [n']^q$ where $i_j
  \neq i_{j'}$ for all $j\neq j'$.
\end{lem}
Thus,
\begin{equation*}
  \esp{X+Y}
  \sim
  \frac{c}{2}+\frac{\lambda_c^2}{4}.
\end{equation*}     

It only remains to examine the higher moments, and show that
\begin{equation*}
  \esp{[X+Y]_k}
  =
  \sum_{k_1+k_2 = k}
  \binom{k}{k_1}
  \sum_{y\in I(k_1,k_2)} \prob{W(y) = 1} 
\end{equation*}
for $y\in I(k_1,k_2)$, where $I(k_1,k_2)$ is the set of tuples $y\in
(\set{e_1,\dotsc, e_\ell})^{k_1}\times (\set{f_1,\dotsc, f_t})^{k_2}$
such that $y_i\neq y_j$ for $i\neq j$ and $\bigcup_{i=1}^k \set{y_i}$
induces a matching on the set of points $S$, and $W(y)$ is the
indicator variable for the event that $X_{i} = 1$ for every $e_i\in
\set{y_1,\dotsc, y_k}$ and $Y_{j} = 1$ for every
$f_j\in\set{y_1,\dotsc, y_k}$.

Let $I'(k_1, k_2)$ be the set of tuples $y\in I(k_1,k_2)$ such that,
in the graph induced by $\bigcup_{i=1}^k \set{y_i}$ in $K$, the degree of
every vertex is either $0$ or $2$. (This is the non-overlapping case.)
Let $I''(k_1, k_2)\eqdef I(k_1,k_2)\setminus I'(k_1,k_2)$.

For $y\in I''(k_1,k_2)$, it is easy to see that the graph induced by
$\bigcup_{i=1}^k\set{y_i}$ in $K$ has more edges than vertices. For
any fixed multigraph $H$ with more edges than vertices, the expected
number of copies of $H$ in $K$ can be bounded as follows. There are at
most $(n')^{|V(H)|}$ ways of assigning the vertices of $H$ to vertices
of $K$. If we assign a vertex with degree $d$ in $H$ to a vertex $v$
in~$K$, then there are at most $\Delta^{d}$ ways of choosing the
points inside $v$ to be the points of the vertex in $H$. So there are
at most $(n')^{|V(H)|} \Delta^{2 |E(H)|} = O((n')^{|V(H)|} (\log n)^{2
  |E(H)|})$ possible copies of $H$ in~$K$. The probability that a set
of $|E(H)|$ edges in $K$ is $O(M^{-|E(H)|})$. Thus, the expected
number of copies of $H$ in $K$ is at most
\begin{equation*}
  O\left( \frac{(n')^{|V(H)|} (\log n)^{2 |E(H)|}}{M^{|E(H)|}}\right)
  =
  O\left( \frac{(n')^{|V(H)|} (\log n)^{2 |E(H)|}}{(n')^{|V(H)|+1}}\right)
  =
  o(1).
\end{equation*}
From this, since $k$ is fixed, we deduce that
\begin{equation*}
  \sum_{k_1+k_2 = k}
  \binom{k}{k_1}
  \sum_{y\in I''(k_1,k_2)} \prob{W(y) = 1}
  = o(1).
\end{equation*}

For $I'(k_1,k_2)$, using~\eqn{delta} and Lemma~\ref{lem:fixed_k},
\begin{equation*}
  \begin{split}
    \sum_{y \in I'(k_1,k_2)}
    \prob{W(y) = 1}
    &=
    \sum_{y \in I'(k_1,k_2)}
    \frac{\delta^{k_2}}{(2M)^{k_1+2k_2}}
    =
    \card{I'(k_1,k_2)}\frac{1}{(2M)^{k_1+2k_2}}
    \delta^{k_2}
    \\
    &=
    \sum_{(v_1,\ldots, v_{k_1+2k_2})}
    \prod_{i=1}^{k_1+2k_2} \binom{d_{v_i}'}{2}
    \frac{1}{(2M)^{k_1+2k_2}}
    \cdot \delta^{k_2}
    \\
    &\sim
    \left(\frac{c}{2}\right)^{k_1+2k_2}\delta^{k_2},
  \end{split}
\end{equation*}
where $v_i\neq v_j$ in $(v_1,\ldots, v_{k_1+2k_2})$ for every $i\neq
j$.

Thus,
\begin{equation*}
  \begin{split}
    \esp{[X+Y]_k}
    &=
    o(1)+
    \sum_{k_1+k_2=k}
    \binom{k}{k_1}
    \sum_{y \in I'(k_1,k_2)}
    \prob{W(y) = 1}
    \\
    &=
    \sum_{k_1+k_2=k}
    \binom{k}{k_1}
    \left(
      \frac{c}{2}
    \right)^{k_1+2k_2}\delta^{k_2}
    +o(1)
    \\
    &=
    \left(
      \frac{c}{2}+\frac{\lambda_c^2}{4}
    \right)^k
    +o(1), 
  \end{split}
\end{equation*}
as required to establish~\eqn{eqn:probA}.

\begin{proof}[Proof of Lemma~\ref{lem:fixed_k}]
  For every $q\geq 1$, let
  \begin{equation*}
    L_q\eqdef 
    \set{(i_1,\dotsc,i_q)\st 1\leq i_j \leq n' \ \forall j};
  \end{equation*}
  \begin{equation*}
    L_q^{\neq}\eqdef 
    \set{(i_1,\dotsc,i_q)\in L_q\st i_j\neq i_{j'} \ \forall j\neq j'}
  \end{equation*}
  and
  \begin{equation*}
    L_q^{=}\eqdef 
    \set{(i_1,\dotsc,i_q)\in L_q\st i_j = i_{j'} \ \text{ for some } j\neq j'}.
  \end{equation*}

  We have 
  \begin{equation*}
    \frac{\sum_{i=1}^{n'} d_i'(d_i'-1)}{\sum_{i=1}^{n'}d_i'}
    =
    \frac{\sum_{i=1}^{n} d_i(d_i-1)-2 D_2}{\sum_{i=1}^{n}d_i-2D_2}
    \sim
    \frac{\bar\eta_ccn -2p_cn}{cn-2p_cn}
    =
    \frac{\bar\eta_cc -2p_c}{c-2p_c}
    =c.
  \end{equation*}
  So, for every $q\geq 1$,
  \begin{equation}
    \label{eq:Lk}
    \sum_{(i_1,\ldots, i_q) \in L_q}
    \prod_{j=1}^q \binom{d_{i_j}'}{2}\cdot\frac{1}{(2M)^q}
    =
    \left( 
      \frac{\sum_{i} \binom{d_i'}{2}}{2M}
    \right)^q
    \sim
    \left(
      \frac{c}{2}
    \right)^q = \Theta(1).
  \end{equation}

    For $q\geq 2$, we have that
  \begin{equation}
    \label{eq:L_equal}
    \begin{split}
      \sum_{(i_1,\ldots, i_q) \in L_q^=}
      \prod_{j=1}^q \binom{d_{i_j}'}{2}\cdot\frac{1}{(2M)^q}
      &\leq
      q!
      \cdot
      \sum_{(i_1,\ldots, i_{q-1}) \in L_{q-1}}
      \binom{d_{i_1}'}{2} \prod_{j=1}^{q-1} 
      \binom{d_{i_j}'}{2}\cdot\frac{1}{(2M)^q}
      \\
      &\leq
      q! \frac{\Delta^2}{4M} 
      \sum_{(i_1,\ldots, i_{q-1}) \in L_{q-1}}
      \prod_{j=1}^{q-1} \binom{d_{i_j}'}{2}\cdot\frac{1}{(2M)^{q-1}}
      \\
      &\sim
      q!\frac{\Delta^2}{4M} 
      \left(\frac{c}{2}\right)^{q-1}
      =
      o(1).
    \end{split}
  \end{equation}

  Note that for $q=1$, we have $L_q = L_q^{\neq}$ and we are done
  by~\eqref{eq:Lk}. So suppose $q\geq 2$. Then $L_q$ is the disjoint
  union of $L_q^{\neq}$ and $L_q^=$. Thus,  using~\eqn{eq:Lk} and
  \eqn{eq:L_equal},
  \begin{align}
    \sum_{(i_1,\ldots, i_q) \in L_q^{\neq}}
    \prod_{j=1}^q &\binom{d_{i_j}'}{2}\cdot\frac{1}{(2M)^q}
    =
    \notag\\&=
    \sum_{(i_1,\ldots, i_q) \in L_q}
    \prod_{j=1}^q \binom{d_{i_j}'}{2}\cdot\frac{1}{(2M)^q}
    -
    \sum_{(i_1,\ldots, i_q) \in L_q^=}
    \prod_{j=1}^q \binom{d_{i_j}'}{2}\cdot\frac{1}{(2M)^q}
    \notag\\
    &=
    \sum_{(i_1,\ldots, i_q) \in L_q}
    \prod_{j=1}^q \binom{d_{i_j}'}{2}\cdot\frac{1}{(2M)^q}
    +o(1)
    \sim
    \left(\frac{c}{2}\right)^q.
    \tag*{\raisebox{-6pt}{\qedhere}}
  \end{align}
\end{proof}


\section{The case $c\to\infty$}

Recall that $n$ and $m$ are such that $m = O(n\log n)$, $m > n$ and
$m/n\to \infty$. The set $\Dcal(n,m)$ contains all degree sequences
$\ds$ such that $\sum_{i=1}^n d_i = 2m$ and $d_i \geq 2$ for all $i\in
[n]$.

Recall that $U(\ds)$ is the probability of obtaining a simple graph
using the pairing model with degree sequence $\ds$, and $U'(\ds)$ is
defined similarly, for the event that it is additionally
$2$-connected.

Let $0 < \eps < 0.01$ be a constant, and let 
\begin{equation*}
  \tilde\Dcal
  \eqdef
  \set{
    \ds\in\Dcal(n,m)\st
    \max d_i \leq n^{\eps}
  }
  \quad\text{and}\quad
  \tilde\Dcal^c\eqdef   \Dcal(n,m)\setminus \tilde\Dcal.
\end{equation*}

By~\cite[Theorem 12.2(iii)]{Luczak}, \bel{UsimUprime} U(\ds) \sim
U'(\ds).  \ee when $\ds$ is in $\Dcal(n,m)$ and satisfies $
D_2(\ds)/m\to 0$ and $\max_i d_i \leq n^{0.01}$. The condition on
$D_2$ is satisfied by all $\ds$ of concern when $n$ is large since
$D_2(\ds)\leq n$ and $c\to\infty$. Thus~\eqn{UsimUprime} holds for any
sequence $\ds(n)$ with $\ds\in\tilde\Dcal$ and $m/n\to\infty$ where
$m=\frac12 \sum_{i=1}^n d_i$.

It is known \cite{McKay} that \bel{mckay} U(\ds)\sim\exp\left( -
  \eta(\ds)/2 - \eta(\ds)^2/4.  \right) \ee This result, together with
\eqn{eqn:pairingd}, proves Theorem~\ref{thm:maindegrees}(c).

If all degree sequences were in $\tilde\Dcal$, we could immediately
deduce Theorem~\ref{thm:maincases}(c) from~\eqn{UsimUprime}. So it
remains to show that the other degree sequences have no effect
asymptotically. We need to randomize $\ds$ with the distribution of
the vector $\Ys$ of independent truncated Poissons as defined
in Section~\ref{ss:models}, we have
\begin{eqnarray}
  \esp[\Big]{U'(\Ys)| \Sigma} 
  &=&
  \esp[\Big]{U'(\Ys)| \tilde\Dcal}
  \prob{\tilde\Dcal|\Sigma} 
  + 
  \esp[\Big]{U'(\Ys)| \tilde\Dcal^c}
  \prob{\tilde\Dcal^c|\Sigma} \nonumber
  \\
  &=& \esp[\Big]{U(\Ys)|
    \tilde\Dcal}(1+o(1))\prob{\tilde\Dcal|\Sigma} 
  +
  O(\prob{\tilde\Dcal^c|\Sigma}) \lab{twoparts}
\end{eqnarray}
by~\eqn{UsimUprime}. Properties of $\Ys$ were investigated by Pittel
and Wormald, and in particular~\cite[Eq. (27)]{PWa} implies for any
$\beta>0$
\begin{equation*}
  \prob{\max_j Y_j \geq m^\beta}
  \leq
  \exp(-n^{\alpha}) 
\end{equation*}
for some fixed $\alpha(\beta)$.  This shows that
$\prob{\tilde\Dcal^c|\Sigma}=O(\exp(-n^{\alpha}))$ for some fixed
positive $\alpha$.  Also,~\cite[Theorem 4(b) and (21)]{PWa} give
\begin{equation*}
  \esp[\Big]{\exp( -\eta(\Ys)/2 -\eta(\Ys)^2/4) |\Sigma}
  \geq
  \exp(-O(\log^2 n)).
\end{equation*}
Using~\eqn{mckay} and the bound on $\prob{\tilde\Dcal^c|\Sigma}$, we
may now deduce that the first term in~\eqn{twoparts} dominates the
second, and thus
$$
\esp[\Big]{U'(\Ys)| \Sigma} 
\sim \esp[\Big]{U(\Ys)|
  \tilde\Dcal}.
$$

Similarly, 
$$
\esp[\Big]{U(\Ys)| \Sigma} 
= 
\esp[\Big]{U(\Ys)|
  \tilde\Dcal}\prob{\tilde\Dcal|\Sigma} 
+
O(\prob{\tilde\Dcal^c|\Sigma})
\sim
\esp[\Big]{U(\Ys)|
  \tilde\Dcal}
$$
and so
\begin{equation}
  \label{eqn:equivexp}
  \esp[\Big]{U'(\Ys)| \Sigma}\sim  \esp[\Big]{U(\Ys)| \Sigma}.
\end{equation}

By Theorem~3 (\cite{PWa}) and equation~(13) (\cite{PWa}),
\begin{equation*}
  \cnm \sim
  (2m-1)!! Q(n,m)
  \esp[\big]{U(\Ys)| \Sigma}
  \sim
  (2m-1)!! Q(n,m) \exp\left( -\bar{\eta}_c/2 - \bar{\eta}_c^2/4 \right).
\end{equation*}

So by~\eqn{eqn:pairingm} and~\eqn{eqn:Qnmdef}, 
\begin{equation}
  \lab{eqn:case2conclusion}
  \tcnm \sim \cnm \sim
  (2m-1)!!
  \frac{(e^{\lambda_c} - 1 - \lambda_c)^n}
  {\lambda_c^{2m}\sqrt{2\pi n c (1 + \bar\eta_c - c)}}
  \exp\left( -\bar{\eta}_c/2 - \bar{\eta}_c^2/4 \right).
\end{equation}
Since $c\to \infty$, we have that $\lambda_c\sim c$ (see Lemma 1(c)
from~\cite{PWa}). This implies that $\bar\eta_c = \lambda_c
e^{\lambda_c} / (e^\lambda_c - 1) \sim c$. This fact together
with~\eqn{eqn:case2conclusion} implies Theorem~\ref{thm:maincases}(c).


\section{Proof of Theorem~\ref{thm:main}}
\label{s:together}

Note that we have already proved Theorem~\ref{thm:maincases}. If we
prove that in each of the three cases in~Theorem~\ref{thm:maincases}
\begin{equation*}
  \tcnm \sim (2m-1)!!
  \frac{(\exp(\lambda_c)-1-\lambda_c)^n}
  {\lambda_c^{2m}\sqrt{2\pi n c(1 + \bar{\eta}_c-c)}} 
  \sqrt{\frac{c-2p_c}{c}} 
  \exp\left(-\frac{c}{2} -\frac{\lambda_c^2}{4} \right) 
\end{equation*}
then the subsubsequence principle easily implies
Theorem~\ref{thm:main}. (See~\cite{JLR} (Section 1.2) for the
subsubsequence principle.)

It suffices to show
\begin{equation*}
  \sqrt{\frac{3r}{2m}} \frac{1}{e}
  \sim 
  \sqrt{\frac{c-2p_c}{c}} \exp\left(
    -\frac{c}{2} - \frac{\lambda_c^2}{4} \right),
  \text{ when }c\to 2
\end{equation*}
and
\begin{equation*}
  \exp\left( -\frac{\bar{\eta}_c}{2}
    -\frac{\bar{\eta}_c^2}{4} \right)
  \sim 
  \sqrt{\frac{c-2p_c}{c}} \exp\left(
    -\frac{c}{2} - \frac{\lambda_c^2}{4} \right),
  \text{ when }c\to \infty.
\end{equation*}
(See~\eqn{csmall} and~\eqn{eqn:case2conclusion}.)

So suppose $c\to 2$. Using Lemma 1 from~\cite{PWa}, $\lambda_c =
3(c-2) + O((c-2)^2) = o(1)$. Thus,
\begin{equation*}
  \exp\left(-\frac{c}{2}-\frac{\lambda_c^2}{4}\right)
  \sim
  \exp\left(-\frac{c}{2}\right)
  \sim
  \frac{1}{e}.
\end{equation*}
By series expansion,
 $ p_c = 1 - \frac{1}{3}\cdot(3(c-2)) + O((c-2)^2)
  = 3-c + O\left( \frac{r^2}{n^2}\right).
$
Using $c = 2m/n = 2+r/n$,
\begin{equation*}
  \begin{split}
    \sqrt{\frac{c-2p_c}{c}}
    &=
    \sqrt{\frac{c-6+2c + O(r^2/n^2)}{c}}
    =
    \sqrt{\frac{3(2+r/n)-6 + O(r^2/n^2)}{2+r/n}}
\sim   \sqrt{\frac{3r}{2m}}.
  \end{split}
\end{equation*}

Now suppose $c\to \infty$.  In this case $\lambda_c\sim c$ (see
Lemma 1(c) from~\cite{PWa}). 
From the definition of $\lambda_c$ we have 
$c = \lambda_c + O(\lambda_c^2 e^{-\lambda_c})$.
Also,
\begin{equation*}
  \bar \eta_c = \lambda_c \cdot \frac{e^{\lambda_c}}{e^{\lambda_c}-1}= 
  \lambda_c + O(\lambda_c e^{-\lambda_c})
  \quad\text{and}\quad
  p_c = \frac{\lambda_c^2}{2(e^{\lambda_c}-1-\lambda_c)}\to 0.
\end{equation*}
This implies
\begin{equation*}
  \sqrt{\frac{c - 2p_c}{c}} \sim 1
  \quad
  \text{and}\quad
  \exp\left(-\frac{\bar\eta_c}{2}-\frac{\bar \eta_c^2}{4}\right)
  \sim
  \exp\left(-\frac{c}{2}-\frac{\lambda_c^2}{4}\right).
\end{equation*}
We now have Theorem~\ref{thm:main}.

\section{Enumeration of $k$-edge-connected graphs}
In the introduction we observed that for $k\ge 3$ and for $m$ under consideration, almost all $k$-cores on $n$ vertices and
$m$ edges are $k$-connected, so it follows that almost all are also $k$-edge-connected. This settles the enumeration of $k$-edge-connected $(n,m)$-graphs for fixed $k\ge 3$. When $k=2$ we have the following result.
\begin{thm}
  Suppose $m = O(n\log n)$ and $ m- n \to \infty$. Then the number of
\tec{}
$(n,m)$-graphs is asymptotic to
  \[  (2m-1)!!
  \frac{(\exp(\lambda_c)-1-\lambda_c)^n}
  {\lambda_c^{2m}\sqrt{2\pi n c(1 + \bar{\eta}_c-c)}} 
  \sqrt{\frac{c-2p_c}{c}} 
  \exp\left(-\frac{c}{2} -\frac{\lambda_c^2}{4} 
            +\frac{\lambda_c^3}{2(e^{\lambda_c} - 1)^2}
  \right).
  \]
\end{thm}
\begin{proof}
 The result is established by adapting the methods used for 2-connected graphs, so we omit unimportant details.

For
$c\to\infty$ we have shown that almost all $(n,m)$-graphs are \tc{}, hence the asymptotic formula for the \tc{} graphs also holds for the \tec{} ones.

For $c\to 2$ our proof actually showed that the probability of a \tec{} simple graph in the kernel configuration model is asymptotic to the probability of a \tc{} simple graph. So, once again, the asymptotic formula for the \tc{} graphs   also holds for the \tec{} graphs.

When $c$ is bounded away from $2$ and bounded, the situation is more interesting.
For $2$-connectivity, the key computation used the method of moments to deduce the Poisson distribution of the number of loops in the kernel plus the number of
double edges in the pseudograph. (See Section~\ref{sec:probcase1a}.) 
 Note that, using~\cite[Lemma~12.1(ii)]{Luczak} as we did in Section 5,  the graph $G$ will a.a.s.\ be  $2$-edge-connected if it has no loops or multiple edges and no cycle in the kernel on a vertex of degree 3. 
Thus, in the present case we must study the random variable $X+Y+Z$, where $X$ counts loops on vertices of degree 3 in the kernel,  $Y$ counts double edges in the kernel  which are assigned no vertices of degree 2, and $Z$ counts loops in the kernel at vertices of degree at least 4 which are not assigned at least 2 degree-2 vertices. Analogous arguments
establish the Poisson distribution of $X+Y+Z$. We discuss only the computation of
the first moment here.

There are three ways to attach a loop to each of the kernel's $D_3$ vertices of degree 3. Analogous to the condition on $D_2(\ds) - p_c n$ in the definition of $ {\tilde\Dcal}(\psi)
$ in Section~\ref{sec:case1a}, we can assume for the crucial computations that  $D_3\sim p_3n$,
where $p_3$ is the probability that
a truncated Poisson $\tpoisson{2}{\lambda_c}$ takes the value 3. 
Each possible loop occurs with probability $1/(2M)$, giving  
$ \esp{X} = D_3 /(2M) \sim c/2 - \lambda_c/2$. 

From Section~\ref{sec:probcase1a} we have $\esp{Y} \sim \lambda_c^2/4$. To
compute $\esp{Z}$ we must first estimate
the probability that a given kernel edge is not assigned at least two degree-2 vertices.
The number of assignments of the $D_2$ degree-2 vertices to the $M$ kernel edges is the rising factorial $[M]^{D_2}$.
Either the given kernel edge is assigned no vertices, which has probability
\[
\frac{ [M-1]^{D_2} }{ [M]^{D_2} } 
= \frac{ M-1 }{ m-1 }
\sim \sqrt{\delta},
\]
or the edge is assigned exactly one vertex, which has probability
\[
D_2 \frac{ [M-1]^{D_2-1} }{ [M]^{D_2} } 
= D_2 \frac{ M-1 }{ (m-2)(m-1) }
\sim (1-\sqrt{\delta})\sqrt{\delta}
\]
since the degree-2 vertex may be chosen in $D_2$ ways. The sum of these two probabilities is $2\sqrt{\delta}-\delta$. 
The number of ways to attach a loop among the vertices of degree at least 4 is $\sum_{i=1}^{n'}\binom{d_i'}{2}-3D_3$. Each occurs with probability $1/(2M)$.
Using Lemma~\ref{lem:fixed_k} we have 
\[
\esp{Z} = \frac{\sum_{i=1}^{n'}\binom{d_i'}{2}-3D_3}{2M}(2\sqrt{\delta}-\delta)
\sim \left(\frac{c}{2}-\frac{3D_3}{2M}\right)(2\sqrt{\delta}-\delta)
\sim \frac{\lambda_c}{2} - \frac{\lambda_c^3}{2(e^{\lambda_c} - 1)^2}.
\]
The probability that $G$ is \tec{} and simple  is thus 
\[
  \exp\left(-\frac{c}{2} -\frac{\lambda_c^2}{4} 
            +\frac{\lambda_c^3}{2(e^{\lambda_c} - 1)^2}
  \right),
\]
and the formula for the number of \tec{} graphs follows as in the \tc{} case. This concludes the proof of the theorem.
\end{proof}
\smallskip

\noindent {\bf Note:}
The alert reader will notice that an alternative way to derive this
result would be to take {\L}uczak's corollary at the end of
Section~12.5 in~\cite{Luczak}, which gives the probability of
$2$-edge-connectedness of graphs with a given degree sequence, and
then use our argument to extend this to graphs with minimum degree
2. The resulting formula agrees with ours if one corrects the formulae
in Theorem~12.4 of his paper, and its Corollary, to let $D_3/M'\to c$
(not $D_3/M$) in his notation.  (We believe the source of this problem
is in the first displayed equation in the proof of~\cite[Theorem
12.4]{Luczak}. The correct definition of $c$ appears just after this
equation.)

\medskip

\noindent {\bf Acknowledgement\ } This work is based on the
foundational work of Boris Pittel with the third author on counting
graphs with given minimum degree, and we gratefully acknowledge
discussions with Boris on the initial stages of the present work.


\providecommand{\bysame}{\leavevmode\hbox to3em{\hrulefill}\thinspace}

\end{document}